\documentclass{amsart}
\usepackage{graphicx}
\usepackage{amssymb}
\usepackage{amsfonts}
\usepackage{hyperref}
\usepackage{mathrsfs}
\swapnumbers
\newtheorem{theorem}{Theorem}[section]
\newtheorem{lemma}[theorem]{Lemma}
\newtheorem{example}[theorem]{Example}

\newtheorem{proposition}[theorem]{Proposition}
\theoremstyle{definition}

\newtheorem{assumption}[theorem]{Assumption}
\newtheorem{remark}[theorem]{Remark}

\numberwithin{equation}{section}
\theoremstyle{plain}

\numberwithin{equation}{section} 
\numberwithin{figure}{section} 
\theoremstyle{plain}
\theoremstyle{plain}
\theoremstyle{remark}
\newtheorem*{acknowledgement*}{Acknowledgement}
\theoremstyle{example}

\newcommand{\cF}{{\mathcal F}}

\newcommand{\cX}{{\mathcal X}}
\newcommand{\cY}{{\mathcal Y}}
\newcommand{\cZ}{{\mathcal Z}}
\newcommand{\te}{{\theta}}

\newcommand{\Om}{{\Omega}}
\newcommand{\om}{{\omega}}
\newcommand{\ve}{{\varepsilon}}

\newcommand{\sig}{{\sigma}}
\newcommand{\al}{{\alpha}}

\newcommand{\ka}{{\kappa}}

\newcommand{\bbE}{{\mathbb E}}

\newcommand{\bbN}{{\mathbb N}}
\newcommand{\bbP}{{\mathbb P}}
\newcommand{\bbR}{{\mathbb R}}

\newcommand{\bbZ}{{\mathbb Z}}
\newcommand{\bbI}{{\mathbb I}}

\begin{document}
\title[]{Effective geometric ergodicity for Markov chains in random environment}  
 \vskip 0.1cm
 \author{Yeor Hafouta}
\address{
Department of Mathematics, The University of Florida, Gainesville, Florida, USA}
\email{yeor.hafouta@mail.huji.ac.il }%

\thanks{ }
\dedicatory{  }
 \date{\today}

\maketitle
\markboth{Yeor Hafouta}{ } 
\renewcommand{\theequation}{\arabic{section}.\arabic{equation}}
\pagenumbering{arabic}

\begin{abstract}
In this short note we prove ``effective" geometric ergodicity (i.e a Perron-Frobenius theorem) for Markov chains in random mixing dynamical environment satisfying a random non-uniform version of the Doeblin condition. 
Effectivity here means that all the random variables involved in the random exponential rates are integrable with arbitrarily large order.  
This complements \cite[Theorem 2.1]{Kifer 1996}, where ``non-effective" geometric ergodicity was obtained. 
From a different perspective, our result is also  motivated by ergodic theory, as  
it can be seen as an effective version of the ``spectral" gap in the top Oseledets space in the Oseledets multiplicative ergodic theorem for the random Markov operator cocycle (when it applies).
We also present applications of the effective ergodicity to rates in the (quenched) almost sure invariance principle (ASIP), exponential decay of correlations for Markovian skew products and for exponential tails for random mixing times. As a byproduct of the proof of the ASIP rates we also
provide easy to verify sufficient conditions for the verification of the assumptions of \cite[Theorem 2.4]{Kifer 1998}. 
\end{abstract}

\section{Introduction}
Let $(X_j)$ be a homogeneous Markov chain and let $R$ be its transition operator. A key tool in studying asymptotic probabilistic properties of the chain is geometric ergodicity, which means that 
$$
\|R^n-\mu\|_{\infty}:=\sup_{g:\|g\|_\infty\leq 1}\|R^ng-\mu(g)\|_\infty\leq C\delta^n
$$
where $\mu$ is the stationary distribution\footnote{Here we view $\mu$ as the linear operator $g\to \mu(g)\textbf{1}$, where $\textbf{1}$ is the function taking the constant value 1. }, $C>0$ and $\delta\in(0,1)$ are constants. Moreover, $\|g\|_\infty=\sup|g|$.

A classical sufficient condition for geometric ergodicity is the, so-called, Doeblin condition which means there exist a probability measure $m$, $n_0\in\mathbb N$ and $\gamma\in(0,1)$ such that $R^{n_0}g\geq \gamma m(g)$ for all bounded measurable functions $g$.

In this paper we consider Markov chains in random dynamical environment. This means that there is an underlying probability space $(\Om,\cF,\bbP)$ and an invertible probability preserving map $\te:\Om\to\Om$ such that for a fixed $\om\in\Omega$ the chain $(X_{\omega,n})_n$ has transition operators of the form $R_n=R_{\te^n\omega}$ when passing from time $n$ to time $n+1$. In this context, geometric ergodicity means that there is a random family of probability measures $\mu_\om$ such that $(R_{\om})^*\mu_\om=\mu_{\te\om}$ and 
$$
\|R_{\te^{-n}\om,n}-\mu_{\om}\|_{\infty}=\sup_{g: \|g\|_\infty\leq 1}\|R_{\te^{-n}\om,n}g-\mu_{\om}(g)\|_{\infty}\leq K(\omega)\delta^n
$$
where $R_{\om,n}=R_{\om}\circ R_{\te\om}\circ\cdots\circ R_{\te^{n-1}\om}$, $K$ is a random variable and $\delta\in(0,1)$. 
Such exponential rates were studied and used in \cite{Cogburn, Kifer 1996, Kifer 1998}. We also refer to \cite{Kifer Thermo,MSU} and references therein for similar results for transfer operators of random expanding dynamical systems. In all these results the random variable $K$ did not satisfy any regularity conditions (like integrability of some order).

One way to obtain \textbf{some} regularity is to apply a version of the Oseledets multiplicative ergodic theorem which ensures that $K(\om)$ is tempered, that is that $K(\te^n\omega)$ grows sub-exponentially fast in $n$ for almost all $\omega$.  Finding sufficient verifiable conditions that ensure better regularity properties of the random variable $K$ is a major problem in the field of random dynamical systems. Regarding the type of regularity, even if we replace the exponential rates with polynomial ones, knowing that $K\in L^p(\bbP)$ is important. We refer to \cite{YH,YH1} for such polynomial effective rates and their applications to various limit theorems. In polynomial  rates we mean that we replace $\delta^n$ above by $n^{-\beta}$ for some $\beta>0$ large enough. Some stretched exponential versions were also proven in \cite{YH}.  The results in \cite{YH,YH1} were obtained under mixing assumptions on $(\Om,\cF,\bbP,\te)$.
Note that the limit theorems in  \cite{YH,YH1} where formulated in a dynamical setup, but as noted in \cite{HafWill} the arguments can be adapted to Markov chains in random environment satisfying a random version of the Doeblin condition. Let us also mention another approach \cite{Gouezel} that under some (relatively strong) assumptions on the base map $(\Om,\cF,\bbP,\te)$ ensures that for $C>0$ large enough the first visiting time to the level set $\{\om: K(\om)\leq C\}$ has sufficiently fast decaying tails. This is also  sufficient to obtain limit theorems by inducing.

 In this paper, under mixing assumptions on $(\Om,\cF,\bbP,\te)$ we prove effective \textbf{exponential} rates. More precisely, under  a random Doeblin condition we prove that there exists $\rho\in(0,1)$ such that for every finite $p\geq 1$ there exists $K_p\in L^p(\bbP)$ such that 
 $$
\max\left(\|R_{\te^{-n}\om,n}-\mu_\om\|_{\infty},\|R_{\om,n}-\mu_{\te^n\om}\|_{\infty} \right)\leq K_p(\omega)\rho^{n/p}.
$$
 As an application we prove quenched almost sure invariance principle rates, exponential decay of correlations for the skew products and exponential tails for the random mixing times. Of course, our results also imply the limit theorems in \cite{HafWill}. As a byproduct of the proofs we also show that under appropriate mixing assumptions on the base map $(\Om,\cF,\bbP,\te)$ we can verify the conditions of \cite[Theorem 2.4]{Kifer 1998}, which seems to be the first time that these conditions are explicitly verified beyond the case of uniform random Doeblin condition.



\section{Preliminaries and effective geometric ergodicity}
Let $Y=(Y_j)_{j\in\bbZ}$ be a  stationary ergodic sequence of random variables taking values on some measurable space $\cY$. Let $(\Omega,\mathcal F,\mathbb P,\te)$ be the shift system generated by this sequence, namely $\Omega=\cY^\bbZ$, $\te:\Omega\to\Omega$ is the left shift and $\bbP$ is the law of the path $(Y_j)_{j\in\bbZ}$. For $-\infty\leq k\leq \ell\leq\infty$,  denote  $\cF_{k,\ell}=\sigma\{Y_s: k\leq s\leq\ell, s\in\mathbb R\}$.  
Recall that the upper $\psi$-mixing coefficients of $Y$ are given by 
$$
\psi_U(n)=\sup_{k}\sup\left\{\frac{\bbP(A\cap B)}{\bbP(A)\bbP(B)}-1:A\in\cF_{-\infty,k}, B\in\cF_{k+n,\infty}, \bbP(A)\bbP(B)>0\right\}.
$$
\begin{assumption}\label{AssMix}
We have $\lim_{n\to\infty}\psi_U(n)=0$.    
\end{assumption}

Next, let $\cX$ be a measurable space and let $\cX_\om$ be measurable in $\omega$ subsets of $\cX$ such that $\cX_\om$ depends only on $\om_0$, where $\om=(\om_n)$. Let $R_\om(x,\Gamma)$ be transition probabilities  which are measurable in $\om$ (here $x\in\cX_
\om$ and $\Gamma\subseteq\cX_{\te\om}$ is a measurable set) and $R_\omega$ depends only on $\om_0$, where again $\om=(\om_n)$.

We assume that there are random variables $n_\om\in\mathbb N$ and $\gamma_\om\in(0,1)$ and a probability measure $m_\om$ on $\cX_\om$, which is measurable in $\omega$ and depends only on $\om_0$,
such that $\bbP$-a.s. for all $x\in\cX_{\te^{-n_\om}\om}$ and a measurable subset $\Gamma\subseteq\cX_\om$ we have
\begin{equation}\label{DoebRand}
R_{\te^{-n_{\om}}\omega,n_\om}(x,\Gamma)\geq \gamma_\om m_{\omega}(\Gamma)
\end{equation}
where $R_{\om,n}=R_{\om}\circ R_{\te\om}\circ\cdots\circ R_{\te^{n-1}\om}$.
Then for every $n\geq n_\omega$, 
$$
R_{\te^{-n}\om,n}(x,\Gamma)\geq \gamma_\om m_{\omega}(\Gamma).
$$
Clearly we can assume that $\gamma_\om$ depends only on $\om_0,\om_{-1},...,\om_{-n_\om}$. Moreover, we can always assume that $n_\omega$ is the minimal positive integer such that \eqref{DoebRand} holds. In this case, the set $\{\omega:n_\omega\leq M\}$ is measurable with respect to $\cF_{-M,0}$ for $M\geq 1$. 

Before proceeding let us describe a non-trivial example when \eqref{DoebRand} holds with a non-constant $n_\omega$.
\begin{example}
Suppose that $Y_0$ takes the values $0$ and $1$ and that $\mathbb P(Y_k=0; 0\leq k\leq m)>0$ for all $m$. For instance, $Y_k$ can be a Markov chain with positive transition probabilities or the $k$-th coordinate of an appropriate subshift of finite type such that $(...,0,0,0,..)$ is an admissible point. Suppose $X_\omega=X$ does not depend on $\om$. Let us consider two Markov  operators $R_0$ and $R_1$ on $X$ and suppose that $R_0$ satisfies the  Doeblin condition, namely there exist a probability measure $m$ on $X$, a constant $\gamma\in(0,1)$ and a positive integer $n_0$ such that 
$$
R_0^{n_0}(x,\Gamma)\geq \gamma m(\Gamma)
$$
for all $x\in X$ and a measurable set $\Gamma$. Set  $R_{\omega}=R_0$ if $\om_0=0$ and $R_{\om}=R_1$ if $\om_0=1$.
Let us take $M(\omega)$ to be the largest negative integer $m<0$ such that $\omega_{k}=0$ for all $m-n_0\leq k\leq m$. Then by the mean ergodic theorem $M(\omega)$ is well defined and finite on a set of probability $1$. Set 
$n_\om=n_0+M(\omega)$. Then for all $n\geq n_\om$, all $x\in X$ and a measurable set $\Gamma\subset X$ we have 
$$
R_{\te^{-n_{\om}}\omega,n_\om}(x,\Gamma)\geq \gamma m(R_{\te^{-M(\omega)}\omega,\omega}(\cdot,\Gamma)).
$$
Thus we can take $\gamma_\omega=\gamma$ and $m_\omega(\Gamma)=m(R_{\te^{-M(\omega)}\omega,\omega}(\cdot,\Gamma))$.
\end{example}
Next, let us give an example where \eqref{DoebRand} holds with a non-constant $\gamma_\omega$.
\begin{example}
Suppose that there is a probability measure $m_\omega$ on $X_\omega$ and densities $p_\omega(x,y)$ bounded below by some positive constant $\zeta_\omega$ such that 
$$
R_{\omega}(x,\Gamma)=\int_{\Gamma}p_\omega(x,y)dm_{\te\om}(y).
$$
Then \eqref{DoebRand} holds with $n_\omega=1$, the above $m_\omega$ and $\gamma_\omega=\zeta_{\theta^{-1}\omega}$. Note that $\gamma_\omega$ may take arbitrary small values.
\end{example}

Henceforth, we will  abuse the notation and write $R_\om g(x)=\int g(y)R_\om(x,dy)$, where $g:\cX_{\te\om}\to\bbR$ is a measurable bounded function, that is we view $R_\om$ as linear operators. 
Denote 
$$
R_{\om,n}=R_{\om}\circ R_{\te\om}\circ\cdots\circ R_{\te^{n-1}\om}=R_{\om_0}\circ R_{\om_1}\circ\cdots\circ R_{\om_{n-1}}.
$$
Our main result in this section is as follows.
\begin{theorem}\label{Main}[Effective random geometric ergodicity]
Under Assumption \ref{AssMix} and \eqref{DoebRand} there exists a unique random family of measures $\mu_\om$ such that  $(R_\omega)^*\mu_\om=\mu_{\te\om}$, for $\bbP$-a.a. $\om$. Moreover, there exists $\rho\in(0,1)$ such that
for every finite $p\geq 1$ there is $K_p\in L^p(\Om,\cF,\bbP)$ such that $\bbP$-a.s. for all $n\in\bbN$,
 \begin{equation}\label{eee}
 \max\left(\|R_{\te^{-n}\om,n}-\mu_\om\|_{\infty},\|R_{\om,n}-\mu_{\te^n\om}\|_{\infty} \right)\leq K_p(\omega)\rho^{n/p}.    
 \end{equation}
\end{theorem}
The proof of Theorem \ref{Main} has two ingredients. The first one is based on a modification of the ideas in \cite{HafWill}, where the random Doeblin condition \eqref{DoebRand} is translated into an explicit upper bound on $\|R_{\te^{-n}\om,n}-\mu_\omega\|_\infty$. Up to a multiplicative constant, this upper bound is a product of certain random variables. In \cite{HafWill}, using certain mixing estimates for expectations of products of sufficiently well mixing random variables we were able to translate these bounds to polynomial or stretched exponential effective mixing rates. Namely, showed that the the left hand side of \eqref{eee} does not exceed $K_p(\omega)a_n$, where  either $a_n=O(n^{-\beta}),\beta>0$ or $a_n=O(e^{-cn^\zeta}), c>0, \zeta\in(0,1)$, depending on the conditions.  Here we prove certain (simple) expectation estimates for products of mixing random variables (see Lemma \ref{KeyLemma}), which allow us to push the argument to effective exponential mixing rates like in \eqref{eee}.

\begin{remark}\label{rrr}
Our proof shows that 
$$
\bbE_\bbP[(K_p(\om))^p]\leq\frac{4C_p\rho^{-2r_0}}{1-\rho^{\frac1{2r_0+2}}}.
$$
Here $r_0$ and $\rho$ are ``computed" as follows. Take $\delta>0$ small enough and $M$ large enough such that the set $A=\{\om: \gamma_\om\geq \delta, n_\om\leq M\}$ has positive $\bbP$ probability, and let $p_0=\bbE[(1-\delta)^{\bbI_A}]\in(0,1)$. Then we take $r_0$ such that 
$$
(1+\psi_U(r_0))p_0<1
$$
and set $\rho=\sqrt{(1+\psi_U(r_0))p_0}$. The constant $C_p$ satisfies $n^{2/p}\rho^{2n}\leq C_p\rho^n$ for all $n$.
This give us some control over the constants in the applications in Sections \ref{App2} and \ref{App3}, and we believe it could also be useful in other applications.
\end{remark}

\section{Proof of the effective geometric ergodicity (Theorem \ref{Main})}
\subsection{A key lemma}
\begin{lemma}\label{KeyLemma}
Let $\beta:\Omega\to[0,1]$ be a random variable which is measurable with respect to $\cF_{-M,0}$ for some $M\in\bbN$. Denote 
$\beta_n(\omega)=\prod_{j=0}^{n-1}\beta(\te^{-Mj}\omega)$.
Suppose that there exists $r_0\in\bbN$ such that 
$$
\rho:=(1+\psi_U(r_0))\bbE[\beta(\cdot)]<1.
$$
Then for every finite $p\geq 1$ there exists a random variable $K_p\in L^p(\bbP)$ such that $\bbP$-a.s. for every $n\in\bbN$,
$$
\beta_n(\om)\leq K_p(\omega)\rho^{\frac{n}{2(r_0+1)p}}.
$$
Moreover, we have 
$$
\|K_p(\cdot)\|_{L^p}^p\leq \frac{\rho^{-3r_0/2-1}}{1-\rho^{\frac{1}{2(r_0+1)}}}.
$$
\end{lemma}
Before proving Lemma \ref{KeyLemma} let us recall the following elementary result which was proven in  \cite[Lemma 60]{YH Adv}, and whose proof proceeds similarly to \cite{NewRef}.

 \begin{lemma}\label{YH Adv}
 Let $I_1,...,I_d$ be intervals in the positive integers so that $I_j$ is to the left of $I_{j+1}$ and the distance between them is at least $L$. Let $A_1,...,A_d$ be nonnegative bounded random variables so that $A_i$ is measurable with respect to $\sig\{Y_k: k\in I_i\}$. Then 
 $$
 \bbE\left[\prod_{i=1}^{d}A_i\right]\leq\left(1+\psi_U(L)\right)^{d-1}\prod_{i=1}^{d}\bbE[A_i].
 $$
 \end{lemma}
 \begin{proof}
Once we prove the lemma for $d=2$ the general case will follow by induction. Let us assume that $d=2$. Next, we have 
$$
A_i=\lim_{n\to\infty}A_i(n)=\lim_{n\to\infty}\sum_k\bbI((k-1)2^{-n}<A_i\leq k2^{-n})k2^{-n}
$$ 
and so with $\al_i(k,n)=\{(k-1)2^{-n}<A_i\leq k2^{-n}\}$, by the monotone convergence theorem we have
$$
\bbE[A_1A_2]=\lim_{n\to\infty}\bbE[A_1(n)A_2(n)]=\lim_{n\to\infty}\sum_{k_1,k_2}(2^{-n}k_1)(2^{-n}k_2)\bbP(\al_1(k,n)\cap\al_2(k,n))$$$$\leq 
\lim_{n\to\infty}\sum_{k_1,k_2}(2^{-n}k_1)(2^{-n}k_2)(1+\psi_U(L))\bbP(\al_1(k,n))\bbP(\al_2(k,n))$$$$=
(1+\psi_U(L))\lim_{n\to\infty}\bbE[A_1(n)]\bbE[A_2(n)]=\left(1+\psi_U(L)\right)\bbE[A_1]\bbE[A_2]
$$
where in the above inequality we have used the definition of the upper mixing coefficients $\psi_U(\cdot)$.
 \end{proof}

\begin{proof}[Proof of Lemma \ref{KeyLemma}]
First, by Lemma \ref{YH Adv} for every $r\geq 2$ we have 
$$
\bbE[\beta_n^p(\omega)]\leq \bbE\left[\prod_{j=1}^{[(n-1)/r]}\beta(\te^{-rMj})\right]\leq (1+\psi_U(r-1))^{[n/r]-1}(\bbE[\beta(\cdot)])^{[n/r]}.
$$
Taking $r=r_0+1$ we see that 
$$
\bbE[\beta_n^p(\omega)]\leq C\rho^{\frac{n}{r_0+1}}
$$
for  $C=\rho^{-r_0-1}$. Now, let
$$
K_p(\omega)=\sup_n\left(\rho^{-\frac{n}{2p(r_0+1)}}\beta_n(\omega)\right).
$$
Then 
$$
\bbE[K_p^p]\leq\sum_{n=1}^\infty\rho^{-\frac{n}{2(r_0+1)}}\bbE[\beta_n^p(\omega)]\leq C\sum_{n=1}^\infty\rho^{\frac{n}{2(r_0+1)}}<\infty. 
$$
\end{proof}

\subsection{Proof of Theorem \ref{Main}}

Let us take $\delta>0$ small enough and $M>0$ large enough such that the set $A=\{\gamma_\omega\geq\delta, n_\om\leq M\}$ has positive probability. Note that  since $\gamma_{\omega}$ depends only on $\om_{0},\om_{-1},...,\omega_{-n_\omega}$ and $n_\omega$ is minimal (see the discussion after \eqref{DoebRand}) $A$ is measurable with respect to $\cF_{-M,0}$.

The proof of Theorem \ref{Main} relies on the following result.
\begin{proposition} 
For $\mathbb P$-a.a. $\om$
 there is a probability measure $\mu_\om$ on $\cX_\om$ such that for all $n\in\mathbb N$,
$$
\|R_{\te^{-n}\om,n}-\mu_\om\|_{\infty}\leq2\prod_{k\leq n: \theta^{-kM}\om\in A}(1-\gamma_{\theta^{-kM}\om}).
$$
In particular,
$$
\|R_{\te^{-n}\om,n}-\mu_\om\|_{\infty}\leq 2(1-\delta)^{\sum_{j=1}^{[n/M]-1}\bbI(\te^{-jM}\om\in A)}.
$$
\end{proposition}
\begin{proof}
The proof uses ideas in the proof of \cite[Corollary 4.1]{HafWill}, but for readers' convenience we provide all the details. Let us fix an $\omega\in\Omega$. We first claim that
 there exist positive measures $A_{\omega,n}$ on $\mathcal X_\omega$, where $n$ satisfies that $\te^{-nM}\omega\in A$, such that $A_{\omega,n}\leq A_{\omega,m}$ when $n<m$ and $\te^{nM}\omega,\te^{mM}\om\in A$ and for $k\geq nM$, and a measurable set $\Gamma\subset \mathcal X_\omega$,
 \begin{equation}\label{Au0}
 \inf_{x\in\mathcal X_{\theta^{-k}\omega}}R_{\te^{-k}\om,k}(x,\Gamma)\geq A_{\omega,n}(\Gamma)    
 \end{equation}
 and
 \begin{equation}\label{Au}
  \sup_{x\in\mathcal X_{\theta^{-k}\omega}}|R_{\te^{-k}\om,k}(x,\Gamma)-A_{\omega,n}(\Gamma)|\leq \prod_{k\leq n: \theta^{-kM}\om\in A}(1-\gamma_{\theta^{-kM}\om}).  
 \end{equation}
 Once this is proven we can take $\mu_\omega=\lim_{n\to\infty}A_{\omega,n}$ (along $n$'s such that $\te^{nM}\omega\in A$) and use that 
 for a transition probability $Q$, a bounded measurable function $g$ and a probability measure $\nu$ we have 
 \begin{equation}\label{lll}
 \left|\int g(y)Q(x,dy)-\int g(y)d\nu(y)\right|\leq 2\sup|g|\sup_{\Gamma}|Q(x,\Gamma)-\nu(\Gamma)|    
 \end{equation}
 where $\Gamma$ ranges over all the underlying measurable sets. In the derivation of \eqref{lll} we used that the total variation distance between two probability measures $\ka_1$ and $\kappa_2$ is given by 
 $$
\|\kappa_1-\kappa_2\|_{TV}=2\sup_{\Gamma}|\kappa_1(\Gamma)-\kappa_2(\Gamma)|
 $$
 and that, in general,
 $$
\|\kappa_1-\kappa_2\|_{TV}=\sup_{\|g\|_\infty\leq 1}\left|\int g\,d\kappa_1-\int g\,d\kappa_2\right|.
 $$

 Next, let us prove the existence of measures $A_{\om,n}$ increasing in $n$ and satisfying \eqref{Au0} and \eqref{Au}. We first note that since $R_{\omega}$ are Markov operators it is enough to prove \eqref{Au} with $k=nM$. In that case the proof proceeds by induction on $n$. When  $n$ is the second time that $\te^{-nM}\om\in A$ then  by the Doeblin condition \eqref{DoebRand}  we can take $A_{\omega,n}(\cdot)=\gamma_{\omega}m_\omega(\cdot)$. Now, let $n$ satisfy $\te^{-nM}\om\in A$, and 
 suppose that there is a measure $A_{\om,n}$ satisfying \eqref{Au0} and that \eqref{Au} holds with $k=nM$.  Let us take the next time $m>n$ such that $\te^{-nM}\om \in A$. 
 Then by the Doeblin condition \eqref{DoebRand} and the induction hypothesis we can write
 $$
 R_{\te^{-mM}\omega,(m-n)M}=\gamma_{\te^{-mM}\omega} m_{\te^{-mM}\omega}+(1-\gamma_{\te^{-mM}\omega})Q
 $$
 and
 $$
 R_{\te^{-nM}\omega,nM}=A_{\omega,n}+\Pi_{\om,n} Q_{\om,n}, \,\,\Pi_{\om,n}=\prod_{k\leq n: \theta^{-kM}\om\in A}(1-\gamma_{\theta^{-kM}\om}).
 $$ 
 Here  $Q(x,dy)$ and $Q_{\om,n}(x',dy)$ are positive transition measures  such that for all relevant points $x$ and $x'$ we have 
 $$
 \max\left(\int Q(x,dy),\int Q_{\om,n}(x',dy)\right)\leq 1.
 $$
Therefore, using also that $\Pi_{\om,m}=(1-\gamma_{\theta^{-mM}\om})\Pi_{\om,n}$,
$$
R_{\te^{-mM}\omega,mM}(x,\Gamma)=\int R_{\te^{-mM}\om,(m-n)M}(x,dy)R_{\te^{-nM}\om,nM}(y,\Gamma)
$$
$$
=\int \left(\gamma_{\te^{-mM}\omega} m_{\te^{-mM}\omega}(dy)+(1-\gamma_{\te^{-mM}\omega})Q(x,dy)\right)(A_{\omega,n}(\Gamma)+\Pi_{\om,n} Q_{\om,n}(y,\Gamma))
$$
$$
=\gamma_{\te^{-mM}\omega}A_{\om,n}(\Gamma)+\gamma_{\te^{-mM}\omega}\Pi_{\om,n}\int Q(y,\Gamma)\,dm_{\te^{-mM}\omega}(y)
$$
$$
+(1-\gamma_{\te^{-mM}\omega})A_{\om,n}(\Gamma)+\Pi_{\om,m}\int Q(x,dy)\,Q_{\om,n}(y,\Gamma)
$$
$$
=A_{\om,n}(\Gamma)+C_{\om,n}(\Gamma)+\Pi_{\om,m}\int Q(x,dy)\,Q_{\om,n}(y,\Gamma)
$$
for some $C_{\om,n}(\Gamma)\geq 0$. Taking $A_{\om,m}(\Gamma)=A_{\om,n}(\Gamma)+C_{\om,n}(\Gamma)$ and using that 
$$
0\leq \int  Q(x,dy)\,Q_{\om,n}(y,\Gamma)\leq 1
$$ 
the proof of the proposition is complete.
\end{proof}

\subsubsection{Completion of the proof of Theorem \ref{Main}}
Define $\beta(\omega)=(1-\delta)^{\bbI(\om\in A)}$. Then $\bbE_\bbP[\beta]<1$. Since $\psi_U(n)\to 0$ we can apply Lemma \ref{KeyLemma} with $r_0$ large enough and find $\rho\in(0,1)$ such that for every $p$ there exists $K_p\in L^p$ with
$$
2(1-\delta)^{\sum_{j=1}^{[n/M]-1}\bbI(\te^{-jM}\om\in A)}\leq K_p(\om)\rho_0^{n/p}
$$
where $\rho_0=\rho^{\frac1{2(r_0+1)M}}$.
This proves the estimate on $\|R_{\te^{-n}\om,n}-\mu_\om\|_{\infty}$ in Theorem \ref{Main}. To prove the estimate on $\|R_{\om,n}-\mu_{\te^n\om}\|_{\infty}$, let $K_p$ be such that 
$$
\|R_{\te^{-n}\om,n}-\mu_\om\|_{\infty}\leq K_p(\om)\rho_0^{n/p}.
$$
Define  $\tilde K_p(\om)=\sup_{n\geq 1}(n^{-2/p}K_p(\te^n\om))$. Then
 $$
\bbE[(\tilde K_p(\om))^p]\leq\|K_p\|_{L^p}^p\sum_{n=1}^\infty n^{-2}<\infty.
 $$
 Thus, 
 $$
\|R_{\om,n}-\mu_{\te^n\om}\|_{\infty}\leq \tilde K_p(\om)(n^{2}\rho_0^n)^{1/p}\leq C_p\tilde K_p(\om)(\rho_0^{1/2})^{n/p}.
 $$
 Thus upon replacing $\rho_0$ with $\rho_0^{1/2}$ the proof of Theorem \ref{Main} is complete.
 \qed

\section{Applications}
\subsection{Application to quenched rates in the almost sure invariance principle (ASIP)}
Let $f:\Omega\times \cX\to\bbR$ be a measurable function and define $f_\omega:\cX_\omega\to\bbR$ by $f_\om(x)=f(\om,x)$. Suppose that $\mu_\om(f_\om)=0$. Let us consider a Markov chain $(X_{\om,j})_j$ such that $X_{\om,j}$ is distributed according to $\mu_{\te^j\om}$ and the $j$-th step transition operator is $R_{\te^j\om}$.
Set
$$
S_n^\om f=\sum_{j=0}^{n-1}f_{\te^j\om}(X_{\om,j}).
$$
\begin{theorem}
Let the assumptions of Theorem \ref{Main} be in force.
Suppose that $F_\om=\|f_\om\|_{\infty}\in L^q(\Om,\cF,\bbP)$ for some $q>2$.
 Then there exists a number $\sigma\geq0$ such that   $\bbP$-a.s. we have 
 $$
\lim_{n\to\infty}\frac1n\bbE[(S_n^\om f)^2]=\sigma^2.
 $$
 If $\sigma>0$ then for $\bbP$ a.a. $\om$ we can couple  the sequence $(S_n^\om)_n$ with a sequence of independent zero mean Gaussian random variables $(Z_n)$ such that for every $\varepsilon>0$,
 $$
\max_{k\leq n}\left|S_k^\om f-\sum_{j=1}^k Z_j\right|=O(n^{1/4+1/q+\varepsilon})
 $$
 and 
 $$
\text{Var}\left(\sum_{j=1}^n Z_j\right)=\text{Var}(S_n^\om f)+O(n^{1/2+1/q+\varepsilon}).
 $$
\end{theorem}

\begin{proof}
To prove the existence of a number $\sigma$ like in the statement of the theorem we will verify the conditions of \cite[Theorem 2.4]{Kifer 1998} with a set of the form $Q_L=\{\om:\max(n_\om,\gamma_\om^{-1})\leq L\}$ for $L$ large enough to ensure that $\bbP(Q_L)>0$. Notice that $Q_L$ is measurable with respect to $\cF_{-L,0}$.
In fact, this will also provide a proof for the CLT and the functional law of iterated logarithm, but these follow from the ASIP.

First, \cite[(2.16)]{Kifer 1998} holds true by \eqref{DoebRand}. Second,  \cite[(2.6)]{Kifer 1998} holds since we are considering functions $f_{\te^n\om}(X_{\om,n})$ of $X_{\om,n}$ and not of the entire path $(X_{\om,n})_{n}$ (and so the approximation coefficients in \cite[(2.6)]{Kifer 1998} vanish).

In order to verify \cite[(2.7)]{Kifer 1998}, let $n_1(\om)$ be the first visiting time to $Q=Q_L$. Denote $c(\om)=\|f_\om\|_{\infty}$. Then it is enough to show that 
$$
\left\|\sum_{j=0}^{n_1(\om)-1}c(\te^j\om)\right\|_{L^2(\bbP)}<\infty.
$$
Next, let us write 
$$
\sum_{j=0}^{n_1(\om)-1}c(\te^j\om)=\sum_{j=0}^\infty c(\te^j\om)\bbI(n_1(\om)>j).
$$
Then by the H\"older inequality,
$$
\left\|\sum_{j=0}^{n_1(\om)-1}c(\te^j\om)\right\|_{L^2(\bbP)}\leq \|c\|_{L^q}\sum_{j=0}^\infty\left(\bbP(n_1>j)\right)^{1/2-1/q}.
$$
Thus, it remains to show that 
\begin{equation}\label{Remains}
\sum_{j=0}^\infty\left(\bbP(n_1>j)\right)^{1/2-1/q}<\infty.    
\end{equation}
To prove that let us notice that for every $r\geq 2$,
$$
\bbP(n_1>j)=\bbP\left(\bigcap_{k=1}^j\te^{-k}(\Omega\setminus Q_L)\right)\leq
\bbP\left(\bigcap_{k=1}^{[j/(rL)]}\te^{-krL}(\Omega\setminus Q_L)\right).
$$
Now, since $Q_L\in\cF_{-L,0}$, by applying Lemma \ref{YH Adv} we see that 
$$
\bbP\left(\bigcap_{k=1}^{[j/(rL)]}\te^{-krL}(\Omega\setminus Q_L)\right)\leq (1+\psi_U(r-1))^{[j/(rL)]}(1-\bbP(Q_L))^{[j/(rL)]}.
$$
Taking $r$ large enough we see that $(1+\psi_U(r-1))(1-\bbP(Q_L))<1$ and thus there exist constants $C>0$ and $\delta\in(0,1)$ such that 
$$
\bbP(n_1>j)\leq C\delta^j
$$
and \eqref{Remains} follows.

Next, let us prove the ASIP rates under the assumption that $\sigma>0$.
Define 
$$
\chi_{\om,n}=\chi_{\om,n}(X_{\om,n})=\sum_{s=n+1}^\infty R_{\te^n\om,s-n}(f_{\te^s\om})=\sum_{s=n+1}^\infty\bbE[f_{\te^s\om}(X_{\om,s})|X_{\om,n}].
$$
Then by Theorem \ref{Main}, for all finite $p>1$,
\begin{equation}\label{chi bound}
\|\chi_{\om,n}\|_{\infty}\leq K_p(\te^n\om)\sum_{s=n+1}^{\infty}\|f_{\te^s\om}\|_{\infty}\rho^{(s-n)/p}.    
\end{equation}
Now, by the ergodic theorem we have $\|f_{\te^s\om}\|_{\infty}=o(s^{1/q})$ and  $K_p(\te^n\om)=o(n^{1/p})$. Therefore, there exists a random variable $C_\omega$ such that 
\begin{equation}\label{chi bound1}
\|\chi_{\om,n}\|_{\infty}\leq C_\om n^{1/p+1/q}+C_\om n^{1/p}\sum_{s=2n}^{\infty}s^{1/q}\rho^{(s-n)/p}=O(n^{1/p+1/q}).    
\end{equation}
Next, we define 
$$
M_{\om,n}=M_{\om,n}(X_{\om,n-1},X_{\om,n})=f_{\te^n\omega}(X_{\om,n})+\chi_{\omega,n}-\chi_{\om,n-1}.
$$
Then $M_{\om,n}$ is a martingale difference. Let $S_n^\om M=\sum_{j=0}^{n-1}M_{\om,j}$.

Now, by taking $p$ large enough in \eqref{chi bound1}, we see that for every $\varepsilon>0$, $\bbP$-a.s. we have
\begin{equation}\label{Apprx1}
\|S_n^\om f-S_n^\om M\|_{\infty}=O(n^{1/q+\varepsilon}).    
\end{equation}
Thus, there exists a random variable $D_\omega$ such that $\|S_n^\om M\|_{L^2}\leq D_\om n^{1/q+\varepsilon}+\|S_n^
\om f\|_{L^2}=O(n^{1/2})$ if $\ve$ is small enough. Therefore both $\|S_n^\om f\|_{L^2}$  and $\|S_n^\om M\|_{L^2}$ are of order $O(n^{1/2})$ and so for $\varepsilon$ small enough,
\begin{equation}\label{Apprx2}
\left|\text{Var}(S_n^\om f)-\text{Var}(S_n^\om M)\right|=O(n^{1/2+1/q+\varepsilon})=o(n).
\end{equation}
Since $\sigma>0$ we get that 
$$
\frac1n\text{Var}(S_n^\om M)\to \sigma^2.
$$
Next, in order to complete the proof of the ASIP we apply \cite[Theorem 2.1]{Shao}. Define $\hat M_{\om,n}(X_{\omega,n-1})=\bbE[(M_{\om,n}(X_{\om,n-1}, X_{\om,n}))^2|X_{\omega,n-1}]-\bbE[(M_{\om,n}(X_{\om,n-1}, X_{\om,n}))^2]$. To verify the conditions of \cite[Theorem 2.1]{Shao} it is enough to prove that 
\begin{equation}\label{(1)}
\sum_{j=0}^{n-1}\hat M_{\om,n}=o(a_n),\,\,\text{a.s.}    
\end{equation}
and 
\begin{equation}\label{(2)}
 \sum_{n\geq 0}a_n^{-2}\bbE[(M_{\om,n})^4] <\infty 
\end{equation}
where $a_n=n^{1/2+2/q+\delta}(\ln n)^{3/2+\delta}, \delta>0$. 
Condition  \eqref{(2)} is in force because of \eqref{chi bound1}, which  implies that $\|M_{\om,n}\|_{L^\infty}=O(n^{1/p+1/q})$ for every finite $p\geq 1$, and our assumption that $\|f_\om\|_{L^\infty}\in L^q(\Om,\cF,\bbP)$.

Next, we verify \eqref{(1)}. For that purpose notice that by conditioning on $X_{\om,m-1}$, then applying Theorem \ref{Main} and using \eqref{chi bound1} for all $k>0$, $m\in\bbN$ and $p>1$ we have 
$$
\left|\bbE[\hat M_{\om,m}(X_{\om,m-1})\hat M_{\om,m+k}(X_{\om,m+k-1})]\right|\leq C_\om (k+m)^{4/q+5/p}\rho^{k/p}.  
$$
Therefore, for all $n,m\in\bbN$
$$
\left\|\sum_{j=m+1}^{m+n}\hat M_{\om,j}(X_{\om,j-1})\right\|_{L^2}^2\leq C_\om (n+m)^{4/q+5/p}n.
$$
Hence \eqref{(1)} follows by \cite[Lemma 9]{[23]} applied with $4/q+5/p$ instead of $p$ (in notations there) and with $\sigma=1$ (in notations there).
\end{proof}

\subsection{Application to exponential decay of correlations for skew products}\label{App2}
Let us denote $\cZ_\om=\prod_{k\in\bbZ}\cX_{\sig^k\om}\subseteq\cZ=\cX^\bbZ$. Consider Markov chains $X_\om:=(X_{\om,k})_{k\in\bbZ}$ such that the law of $X_{\om,k}$  is $\mu_{\te^k\om}$ and the transition probabilities are $R_{\te^k\om}$. Such chains are well defined since  $(R_\om)^*\mu_{\om}=\mu_{\te\om}$. Let $\kappa_\omega$ be the law of $X_\omega$ on $\cZ_\om$. 
 Let us define the skew product sequence $Z_n(\omega,z)$ by 
$$
Z_n(\om,z)=(\te^n\omega,z_n), z=(z_n)\in\cZ.
$$
Let us view $Z_n$ as a sequence of random variables 
with respect to the measure $\kappa=\int\kappa_\om d\bbP(\omega)$. Now, define the projection onto the $n$-th coordinate $\pi_n:\Omega\times\cZ\to\cY\times\cX$ by
$$
\pi_n(\om,z)=(\om_n,z_n).
$$
 Then $\pi_n=\pi_0\circ Z_n$.

 Recall next that the $\rho$-mixing coefficient of the process $(Y_j)$ is given by 
 $$
 \rho(n)=\sup_k\sup\left\{|\text{Corr}(f,g)|: f\in L^2(\cF_{-\infty,k}), g\in L^2(\cF_{k+n,\infty})\right\}.
 $$
\begin{theorem}
Let the assumptions of Theorem \ref{Main} be in force, and let $\rho\in(0,1)$ be as described in that theorem.
Let $f_{\om_0}(x)=f(\om_0,x)$ and $g_{\om_0}(x)=g(\om_0,x)$ be two measurable functions on $\cY\times \cX$.  
Suppose that $F_\om=\|f_{\om_0}\|_{\infty}$ and  $G_\om=\|g_{\om_0}\|_{\infty}$ are in $L^{q}(\Om,\cF,\bbP)$ for some $q>2$. Assume also that $\rho(n)=O(\rho_2^n)$ for some $\rho_2\in(0,1)$. Then there exists a constant $C>0$ such that 
$$
\left|\text{Cov}((f\circ\pi_0),(g\circ \pi_n))\right|\leq C\rho_3^n\|F_\om\|_{L^q}\|G_\om\|_{L^q}
$$
where $\rho_3=\max(\rho_2^{1/2},\rho^{1/(2p)})$ with $p=\frac{2q}{q-2}$.
\end{theorem}

\begin{proof}
By normalizing $f$ and $g$ if needed, it is enough to prove the theorem when $\max(\|F_\om\|_{L^q},\|G_\om\|_{L^q})\leq 1$.
Write
$$
\bbE_{\kappa}[(f\circ\pi_0)\cdot (g\circ \pi_n)]=\int_{\Omega}\bbE[f(\om_0,X_{\om,0})g(\om_n,X_{\om,n})]d\bbP(\om)
$$
$$
=\int_{\Omega}\bbE[f_{\om_0}(X_{\om,0})R_{\om,n}g_{\om_n}(X_{\om,n})]d\bbP(\om)
=\int_{\Omega}\mu_\om(f_{\om_0})\mu_{\te^n\om}(g_{\om_n})d\bbP(\om)+O(\rho^{n/2}),
$$
where the last equality uses Theorem \ref{Main} and that $\int \|f_{\om_0}\|_\infty \|g_{\om_n}\|_\infty d\mathbb P(\omega)\leq 1$.

Next, by Theorem \ref{Main} we see that 
$$
\|\mu_{\te^n\om}(g_{\om_n})-R_{\te^{[n/2]}\om}^{n-[n/2]}g_{\om_n}\|_{\infty}\leq \|g_{\om_n}\|_{\infty}K_p(\te^{[n/2]}\om)\rho^{n/(2p)}.
$$
Denote $G_n(\om)=\mu_{\te^n\om}(g_{\om_n})$. Fix some $p>2$ and let  $s$ be given by $1/s=1/p+1/q$, where $q$ comes from the assumptions of the theorem. Note that $R_{\te^{[n/2]}\om}^{n-[n/2]}g_{\om_n}$ is measurable with respect to $\cF_{n-[n/2],n}$. Thus, 
 by the previous estimate and the minimization property of conditional expectations,
\begin{equation}\label{r approx}
\|G_n-\bbE[G_n|\cF_{n-[n/2],n}]\|_{L^s}\leq\left\|\|g_{\om_n}\|_\infty\right\|_{L^q}\|K_p\|_{L^p}\rho^{n/(2p)}\leq  C_{p}\rho^{n/(2p)}
\end{equation}
for some constant $C_{p}$ that does not depend on $n$. In \eqref{r approx} we also used that 
$\left\|\|g_{\om_n}\|_\infty\right\|_{L^q}\leq1$.
Next, using that $\mathbb P$ is $\te-$invariant, we have
$$
\int_{\Omega}\mu_\om(f_{\om_0})\mu_{\te^n\om}(g_{\om_n})d\bbP(\om)-
\int_{\Omega}\mu_\om(f_{\om_0})d\bbP(\om)\int_{\Omega}\mu_{\om}(g_{\om_0})d\bbP(\om)
$$
$$
=
\int_{\Omega}\mu_\om(f_{\om_0})G_n(\om)d\bbP(\om)-
\int_{\Omega}\mu_\om(f_{\om_0})d\bbP(\om)\int_{\Omega}G_n(\omega)d\bbP(\om).
$$
Now take  $s=2$ and let $p$ be given by $1/2=1/q+1/p$ (that is, $p=\frac{2q}{q-2}$).  
Then by \eqref{r approx} and using that $\int \|f_{\om_0}\|_\infty^2 d\mathbb P(\omega)\leq 1$ we get that
$$
\int_{\Omega}\mu_\om(f_{\om_0})G_n(\om)d\bbP(\om)=\int_{\Omega}\mu_\om(f_{\om_0})\bbE[G_n|\cF_{n-[n/2],n}]d\bbP(\om)+O(\rho^{n/(2p)}).
$$
Now, by the  definition of the $\rho$ mixing coefficients and since $\|G_n\|_{L^2}, \|F_{\cdot}\|_{L^2}\leq1$ we have
$$
\int_{\Omega}\mu_\om(f_{\om_0})\bbE[G_n|\cF_{n-[n/2],n}]d\bbP(\om)
=
\int_{\Omega}\mu_\om(f_{\om_0})d\bbP(\om)\int_{\Omega}\bbE[G_n|\cF_{n-[n/2],n}]d\bbP(\om)+O(\rho_2^{n/2}).
$$
Finally, using again \eqref{r approx} with $s=2$ and that $\int_{\Omega}|\mu_\om(f_{\om_0})|d\bbP(\omega)\leq1$,
$$
\int_{\Omega}\mu_\om(f_{\om_0})d\bbP(\om)\int_{\Omega}\bbE[G_n|\cF_{n-[n/2],n}]d\bbP(\om)=\int_{\Omega}\mu_\om(f_{\om_0})d\bbP(\om)\int_{\Omega}G_n(\omega)d\bbP(\om)+O(\rho^{n/(2p)}).
$$
 \end{proof}
 \begin{remark}
 We expect to get results for functions $f,g$ which depend on the entire orbit $\om_j$ and can be approximated exponentially fast in $r$ by functions of $\om_j, |j|\leq r$ in an appropriate sense, but we decided to formulate the result in the a simpler situation in order to avoid a heavy notation.    
 \end{remark}

\subsection{Applications to random mixing times}\label{App3}
Given $\ve>0$ we define 
$$
N_\ve(\om)=\min\{n\in\bbN: \|R_{\om,n}-\mu_{\te^n\om}\|_{\infty}\leq \ve\}.
$$ 
\begin{assumption}\label{a}
Suppose that $\bbP$-a.s. for all $n\in\bbN$,
   $$
\|R_{\om,n}-\mu_{\te^n\om}\|_{\infty}\leq K(\omega)\rho^{n/p}
    $$
    where $K(\om)\in L^p$ and $\rho\in(0,1)$.
\end{assumption}

\begin{theorem}
Under Assumption \ref{a} for every $\ve>0$ we have 
$$
\bbP(\om: N_\ve(\om)>N)\leq \|K\|_{L^p}^p\ve^{-p}\rho^{N}.
$$
\end{theorem}
\begin{remark}
We can take $\ve=\varepsilon_N=\rho^{\frac{N}{2p}}$ and get that
$$
\bbP(\om: N_{\ve_N}(\om)>N)=O(\rho^{N/2}).
$$
Note that in our circumstances we can take $p$ arbitrarily large.
\end{remark}

\begin{proof}
Denote $a_n=\rho^{n/p}$.
We have 
$$
\{\om: N_\ve(\om)>N\}=\{\om: \|R_{\om,n}-\mu_{\te^n\om}\|_{\infty}\geq \ve,\,\,\forall n\leq N\}
\subseteq\{\om: K(\om)a_n\geq \ve,\,\,\forall n\leq N\}
$$
$$
=\{\om: K(\omega)\geq \ve a_N^{-1}\}.
$$
Thus, the result follows by the Markov inequality.
\end{proof}

\end{document}